\documentclass[12pt, A4paper]{article}

\usepackage{amssymb}
\usepackage{latexsym}
\usepackage{amsmath}
\usepackage{amsthm}
\usepackage{multicol}

\usepackage[hmargin=3cm,vmargin=3.5cm]{geometry}

\usepackage{graphicx}
\input xy
\input diagxy
\xyoption{all}

\newcommand{\bbB}[0]{{\mathbb B}}
\newcommand{\bbS}[0]{{\mathbb S}}
\newcommand{\bbC}[0]{{\mathbb C}}
\newcommand{\bbD}[0]{{\mathbb D}}

\newcommand{\subsetdot}[0]{\mathrel{\dot{\subseteq}}}
\newcommand{\equaldot}[0]{\mathrel{\dot{=}}}

\newcommand{\dotin}[0]{ \,\dot{\in}\,}

\newcommand{\refl}[0]{{\rm ref}}

\newcommand{\omitthis}[1]{}

\newcommand{\longtext}[1]{}
\newcommand{\shorttext}[1]{}
\newcommand{\commentaway}[1]{}

\newtheorem{theorem}{Theorem}[section]

\newtheorem{proposition}[theorem]{Proposition}

\newtheorem{example}[theorem]{Example}

\begin{document}

\title{Yet another category of setoids \\ with equality on objects}
\author{Erik Palmgren \\ Department of Mathematics, Stockholm University}
\date{April 21, 2013}

\maketitle

\begin{abstract} When formalizing mathematics in (generalized predicative) constructive type theories, or more practically in proof assistants such as Coq or Agda,  one is often using setoids (types with explicit
equivalence relations). In this note we consider two categories of  setoids with equality on objects and show that they are isomorphic. Both categories are constructed from a fixed proof-irrelevant family $F$ of setoids. The objects
of the categories are the index setoid  $I$ of the family, whereas the definition of arrows differs. The first category has for arrows triples $(a,b,f:F(a) \to F(b))$ where $f$ is an extensional function. Two such arrows are identified if appropriate 
composition with transportation maps (given by $F$) makes them equal. In the second category the arrows are triples $(a,b,R \hookrightarrow \Sigma(I,F)^2)$ where $R$ is a total functional relation between the subobjects $F(a), F(b) \hookrightarrow \Sigma(I,F)$ of the setoid sum of the family. This category is simpler to use as the transportation maps disappear. Moreover we also show that the full image of a category along an E-functor into an E-category is category.

\end{abstract}

\section{Introduction}

In type theory there is a choice whether categories should be equipped with an equality on objects or not. Categories without
such equality are called E-categories. For some purposes equality on objects seem necessary, but here one already encounter
problems to construct rich categories of setoids. (Setoids are types with equivalence relations, the natural notion of set in type theory).
From any proof-irrelevant family of setoids $F$ over a setoid $A$  arises a category ${\cal C}(A,F)$ of setoids 
as shown in \cite{W,PW}. This is perhaps the category which is closest
to the type-theoretic language: the setoid $A$ constitutes the objects,
a morphism $(a,b,f)$ consists of an extensional function $f:F(a) \to F(b)$ where $ a, b \in A$.
Composition with another morphism $(b',c,g)$ is possible if $b=_A b'$ and is defined using a 
transportation function $F(p):F(b) \to F(b')$ associated with the proof $p:b=_A b'$. 
Equality of two morphisms $(a,b,f)$ and $(a',b',f')$  is taken to mean that there are proofs $p:a=_Aa'$ and $q:b=_Ab'$ such that
$$\bfig\square[F(a)`F(b)`F(a')`F(b');f`F(p)`F(q)`f']\efig$$
commutes. 
This equality is a slightly cumbersome notion when working with this category.
We show (Theorem \ref{mainiso}) that this category is isomorphic to another category ${\cal S}(A,F)$  where the morphisms corresponds to 
functional relations on $\Sigma(A,F)$ the setoid-sum of the family. In this category the transportation functions are invisible, which
makes for a smoother treatment of the category of setoids akin to a category of sets in set theory. As shown in \cite{PW} we may chose $F$ so that
 ${\cal C}(A,F)$ is isomorphic to a category of sets in a model of constructive set theory CZF, thus ensuring rich properties of the category.

\section{Families of setoids} 

Recall from, for instance \cite{P12} or \cite{PW}, that a good notion of a family of setoids over a setoids is the following.
A {\em proof-irrelevant family $F$ of setoids over $A$} --- or just {\em family of setoids} --- consists of a setoid
$F(x) = (|F(x)|, =_{F(x)})$  for each $x\in A$, 
and for $p: (x=_A y)$ an extensional function $F(p): F(x) \to F(y)$ (the transportation function)
which satisfies the three conditions: 
\begin{itemize}
\item[(F1)] $F(\refl(x))=_{\rm ext} {\rm id}_{F(x)}$ for  $x \in A$. Here $\refl(x)$ is the canonical proof object for $x=_A x$ and $=_{\rm ext}$ denotes the extensional equality of functions.

\item[(F2)] $F(p) =_{\rm ext} F(q)$ for $p,q : x=_A y$ and $x,y \in A$. Since
$F(p)$ does not depend on $p$, this is the proof-irrelevance condition.

\item[(F3)] $F(q) \circ F(p) =_{\rm ext} F(q \circ p)$ for 
$p: x =_A y$, $q: y =_A z$ and $x,y,z \in A$.
\end{itemize}

\medskip
Proof-irrelevant families may arise as functions $I \to {\rm P}(A)$ from the index setoid $I$ 
into the collection of subsetoids of a fixed setoid $A$ in the following way.
Let $A$ be the fixed setoid. Let ${\rm P}(A)$ denote the following
preorder. Its elements are injections $m: U \to A$, where $U$ is setoid. Let $n: V \to A$ be another injection. We say that it {\em includes}  $m: U \to A$, in symbols $(U,m) \subsetdot(V,n)$, if
there is a function $k:U \to V$ such that $n \circ k = m$. (Note that $k$ is unique and an injection.) Now define
 $m: U \to A$  and  $n: V \to A$ to be {\em equal},  or in symbols $(U,m) \equaldot (V,n)$,  if 
$(U,m) \subsetdot (V,n)$ and $(V,n) \subsetdot (U,m)$. Thus ${\rm P}(A)$  has an equivalence relation. 
Indeed, defining for $x \in A$ and $(U,m) \in {\rm P}(A)$, a membership relation
$$x \dotin (U,m) \Longleftrightarrow_{\rm def} (\exists u \in U)x=_A m(u),$$
we get using unique choice
$$(U,m) \subsetdot (V,n) \mbox{ iff } (\forall x \in X)(x \,\dot{\in}\, (U,m) \Rightarrow x \,\dot{\in}\, (V,n)).$$
Thus we see that $(U,m) \equaldot (V,n)$ is extensional equality.

A {\em family of subsetoids of $A$ indexed by a setoid $I$} is an extensional function $F:I \to {\rm P}(A)$.  Write $F(i) = (\hat{F}(i), m_i)$. We may now extend $\hat{F}$ to a proof-irrelevant family in a canonical way: for a proof $p$ of $i=_I j$, we have 
$F(i) \equaldot F(j)$  so there is a  unique $f$ such that the following diagram commute
$$\bfig
\Vtriangle(0,0)/>`{ >}->`{ >}->/[\hat{F}(i)`\hat{F}(j)`A;f`m_i`m_j]
\efig
$$
We let $\hat{F}(p) =_{\rm def}f$. By the above it is unique and independent of $p$, so (F2) holds. If $i=j$ definitionally, then $f$ is extensionally
equal to the identity on $\hat{F}(i)$. This verifies (F1). The condition (F3) of $\hat{F}$ is easy to check using uniqueness. 

Conversely, from every proof-irrelevant family $F$ on $I$ we get  a family $\check{F}: I \to {\rm P}(A)$ for a canonical $A$; see Proposition \ref{pif2fsub} below. To prove this we introduce the setoid-sum construction.
%
Let  $F$ be a family of setoids over the setoid $I$.
The {\em disjoint sum of the family} exists in {\bf Setoids} and may be explicitly given by
$$\Sigma(I,F) =_{\rm def} \bigl((\Sigma x :|I|)|F(x)|, \sim\bigr)$$
where the equivalence is given by
$$\mbox{$(x,y) \sim (x',y')$ iff $(\exists p:x=_I x')(F(p)(y) =_{F(x')}y')$.}$$ 
The injections
$$F(x) \to^{\iota_x} \Sigma(I,F)$$
are given by $\iota_x(y) =(x,y)$, and satisfy 
\begin{equation} \label{injprop}
\iota_{x'} \circ F(p) =_{ext} \iota_x \text{ for } p:x=_I x'.
\end{equation}
This construction satisfies the {\em universal property} that if $C$ is a setoid and $j_x: F(x) \to C$ ($x \in I$) are extensional functions
with $j_{x'} \circ F(p) =_{ext} j_x$ for all $p:x =_A x'$, then there is a unique extensional $k: \Sigma(I,F) \to C$ with  $k \circ \iota_x =_{ext} j_x$
for all $x \in I$. 


\begin{proposition} \label{pif2fsub}
Let $F$ be a family of setoids indexed by the setoid $I$. Then $F$ induces 
an extensional  function 
$$\check{F}: I \to {\rm P}(\Sigma(I,F))$$
where $\check{F}(x) = (F(x), \iota_x)$ and $\iota_x: F(x) \to \Sigma(I,F)$ is the canonical injection.
\end{proposition}
\begin{proof} It follows from (\ref{injprop}) that $p: x=_I y$, implies 
$\check{F}(x) \subsetdot \check{F}(y)$ and similarly $p^{-1}:y=_I x$ implies
$\check{F}(y) \subsetdot \check{F}(x)$. Thus $\check{F}(x) \equaldot \check{F}(y)$. 

\end{proof}

\section{Two categories of setoids and their isomorphism}

We provide some more details to the construction sketched in the introduction; see \cite{PW} for full details.
A  family $F$ of setoids over a setoid $I$ gives rise to a category of setoids   
${\cal C}={\cal C}(I,F)$ as follows. The objects are given by
the index setoid ${\cal C}_0=I$, and are thus equipped with equality, and the
setoid of arrows ${\cal C}_1$  is
$$((\Sigma i,j: |I|){\rm Ext}(F(i),F(j)),\sim)$$ 
which, thus, consists of triples $(i,j,f)$ where $f:F(i) \to F(j)$ is an extensional function,
and where
two arrows are equal $(i,j,f) \sim (i',j',f')$ if, and only if, there are proof objects $p: i=_I i'$ and $q: j=_I j'$  such that
the diagram
\begin{equation} \label{ijfdiagram}
\bfig\square[F(i)`F(j)`F(i')`F(j');f`F(p)`F(q)`f']\efig
\end{equation}
commutes.
The domain of the arrow $(i,j,f)$ is $i$ and its codomain is $j$. Arrows  $(i,j,f)$ and  $(j',k,g)$ are composable if there is $p:j=_I j'$ and
their composition is $(i,k, g \circ F(p) \circ f)$. Note that $F(p)$ and hence the composition is independent of $p$.  The setoid ${\cal C}_2$ of composable
arrows consists of such triples $((i,j,f),(j',k,g),p)$. Then ${\cal C}$ is a category in the essentially algebraic sense as shown in \cite{PW}. (See Appendix
for the formal definition of essentially algebraic category.)

The second construction is as follows.
Define a category ${\cal S}(I,F)$ whose setoid of objects is $I$, and whose arrows\footnote{The triples actually form a setoid since they can be represented by graphs of functions, as the isomorphism theorem shows later.} are triples $(i,j,R)$ where 
$R$ is  functional binary relation on $S=\Sigma(I,F)$ with ${\rm dom}(R) \equaldot F(i)$ and ${\rm ran}(R) \subsetdot F(j)$. Two
arrows $(i,j,R)$ and $(i',j',R')$ are equal when $i=_I i'$, $j=_I j'$ and $R \equaldot  R' $.
The domain and codomain of $(i,j,R)$ are $i$ and $j$ respectively. The composition of $(i,j,R)$ and $(j',k,Q)$ is
$(i,k,Q \circ R)$ when $j=_I j'$. Here $Q \circ R$ denotes the relational composition.

Now define a functor $M:{\cal C}(I,F) \to {\cal S}(I,F)$ by letting $M$ be the identity on objects, $M(i)=i$, and  for an arrow $(i,j,f)$
letting $M(i,j,f)= (i,j,{\cal G}_f)$ where ${\cal G}_f$ is the {\em graph of $f$} on $S\times S$ defined
by
\begin{equation} \label{graphdef}
(u,v) \dotin  {\cal G}_f \Longleftrightarrow_{\rm def} (\exists x \in F(i))[u=_{S} \langle i, x\rangle \land v=_{S}  \langle j, f(x)\rangle ]
\end{equation}
We show that $M$ is well-defined on arrows: Suppose that  $(i,j,f)$ and  $(i',j',f')$ are equal arrows in ${\cal C}(I,F)$, that is, there are $p:i=_I i'$ and $q:j=_J j'$ such that the diagram (\ref{ijfdiagram})
commutes. Note that for $x \in F(i)$,  $\langle i,x \rangle =_S \langle i',F(p)(x)\rangle$ and $\langle j,f(x) \rangle = \langle j',F(q)(f(x)) \rangle$.
Inserting this in  (\ref{graphdef}), substituting $x=F(p^{-1})(x')$  and then using the commutative square we get
\begin{eqnarray*}
(u,v) \dotin  {\cal G}_f &\Longleftrightarrow& (\exists x \in Fi)[u=_{S}\langle i',F(p)(x)\rangle \land v=_{S}   \langle j',F(q)(fx) \rangle ]\\
&\Longleftrightarrow& (\exists x' \in Fi')[u=_{S}\langle i',F(p)(F(p^{-1})(x')))\rangle \land \\
&& \qquad \qquad \qquad  
 v=_{S}   \langle j',F(q)(f(F(p^{-1})(x'))) \rangle ] \\
&\Longleftrightarrow& (u,v) \dotin  {\cal G}_{f'} 
\end{eqnarray*}
Thus $M$ is well-defined.

For objects $(i,j,f)$ and $(j',k,g)$ with $p:j=_Ij'$ we check functoriality by verifying that
\begin{equation} \label{graphfunctorial}
{\cal G}_g \circ {\cal G}_f \equaldot {\cal G}_{g \circ F(p) \circ f}
\end{equation}
and that
\begin{equation} \label{graphfunctorial2}
{\cal G}_{{\rm id}_{F(i)}}
\end{equation}
is the identity relation on the subset $\check{F}(i)$. To see (\ref{graphfunctorial}) expand the definition
and use that $F(p)=F(q)$ for $q:j=_Ij'$:
\begin{eqnarray*} 
(*) \qquad(u,v) \dotin  {\cal G}_g \circ {\cal G}_f 
&\Longleftrightarrow&(\exists x \in Fi)(\exists y \in Fj') \\
&& \qquad \qquad  (u=\langle i,x\rangle \land\langle j',y\rangle= \langle j,fx \rangle \land v= \langle k, gy)) \\
&\Longleftrightarrow&(\exists x \in Fi)(\exists y \in Fj')(\exists q:j=_I j') \\
&& \qquad \qquad  (u=\langle i,x\rangle \land   F(q)(fx) =_{Fj'} y  \land v= \langle k, gy)) \\
&\Longleftrightarrow&(\exists x \in Fi)(\exists y \in Fj') \\
&& \qquad \qquad  (u=\langle i,x\rangle \land   F(p)(fx) =_{Fj'} y  \land v= \langle k, gy)) \\
&\Longleftrightarrow&(\exists x \in Fi) \\
&& \qquad \qquad  (u=\langle i,x\rangle \land v= \langle k, g(F(p)(fx))\rangle) \\
&\Longleftrightarrow& (u,v) \dotin  {\cal G}_{g \circ F(p) \circ f} 
\end{eqnarray*}

Further
\begin{eqnarray*}
(**) \qquad (u,v) \dotin  {\cal G}_{{\rm id}_{Fi}} &\Longleftrightarrow & (\exists x \in Fi)[u=_{S} \langle i, x\rangle \land v=_{S}  \langle i, {\rm id}_{Fi}(x))\rangle ] \\
&\Longleftrightarrow & u=_{S} v \land (\exists x \in Fi)u=_{S} \langle i, x\rangle
\end{eqnarray*}
which is the identity relation on $\check{F}(i)$. Call this relation $I_{\check{F}(i)}$ for later use.

\medskip
Define in the opposite direction a functor $N:{\cal S}(I,F) \to\, {\cal C}(I,F)$ by letting it be the identity on objects, 
and for a morphism
 $(i,j,R)$ let $f:F(i) \to F(j)$ be the
unique extensional function such that 
\begin{equation} \label{GfR}
{\cal G}_f \equaldot R.
\end{equation}
Let $$N(i,j,R)=_{\rm def} (i,j,f).$$
Existence of $f$: Suppose $(i,j,R)$ is morphism. Hence 
$$(\forall x \in Fi)(\exists! y \in Fj)(\langle  i,x \rangle,\langle j,y \rangle) \dotin R.$$
Thus there is a unique extensional  $f: F(i) \to F(j)$ such that
\begin{equation} \label{graphed}
(\forall x \in Fi)(\langle  i,x \rangle,\langle j,f(x)\rangle) \dotin R.
\end{equation}
If $(u,v) \dotin {\cal G}_f$, then by (\ref{graphdef}) there is $x \in F(i)$ such that
$$u=_{S} \langle i, x\rangle \land v=_{S}  \langle j, f(x)\rangle.$$
Thus by (\ref{graphed}): $(u,v) \dotin R$.
Conversely, suppose  $(u,v) \dotin R$. Then since ${\rm dom}(R) = \check{F}(i)$ and 
${\rm ran}(R) \subsetdot \check{F}(j)$, there is $x \in F(i)$
 and $y \in F(j)$, with $u=_{S} \langle i, x\rangle$ and $v=_{S}  \langle j, y\rangle$.
 By uniqueness in (\ref{graphed}),  $y=f(x)$, so indeed  $(u,v) \dotin {\cal G}_f$. Thus (\ref{GfR})
 holds.
 
 Uniqueness of $f$: Suppose that ${\cal G}_{f'} \equaldot R$ for some $f':F(i) \to F(j)$. Then
 $(\forall x \in Fi)(\langle  i,x \rangle,\langle j,f'x\rangle) \dotin R.$ By uniqueness in
 (\ref{graphed}), $f'=f$.
 
 We show $N$ is well-defined on arrows: Suppose $(i,j,R)$ and $(i',j',R')$ are equal morphisms with $N(i,j,R) =(i,j,f)$ and $N(i',j',R')=(i',j',f')$.
 Thus $p:i=_I i'$ and $q:j=_J j'$ and $R \equaldot R'$, and hence
 $${\cal G}_f  \equaldot {\cal G}_{f'}.$$
 We show that (\ref{ijfdiagram}) commutes. Let $x \in F(i)$. Then by definition of the graph ${\cal G}_f $, we get
 $(\langle i, x \rangle,\langle j, f(x) \rangle)\dotin {\cal G}_f $, and hence also
 $(\langle i, x \rangle,\langle j, f(x) \rangle)\dotin {\cal G}_{f'} $. Again by the definition of graph:
$$ (\exists x' \in Fi')[\langle i, x \rangle=_{S} \langle i', x'\rangle \land \langle j, fx \rangle =_{S}  \langle j', f'x'\rangle ].$$
Thus for some $x' \in Fi'$, $p':i=_I i'$ and some $q':j=_I j'$ we have
$$F(p')(x)=_{Fi'} x' \qquad F(q')(fx) =_{Fj'} f'x'.$$
Hence
$$F(q')(fx) =_{Fj'} f'(F(p')(x)),$$
and since $F(q)=F(q')$ and $F(p) = F(p')$, we are done proving that the diagram commutes.

We check that $N$ is functorial: Suppose that $N(i,j,R) =(i,j,f)$ and $N(j',k,Q)=(j',k,g)$ with $p:j=_I j'$. Then
$$N(j',k,Q) \circ N(i,j,R)= (i,k, g \circ F(p) \circ f).$$
Now
$$N((j',k,Q) \circ (i,j,R)) = N(i,k, Q \circ R)=(i,k,h)$$
where $h:F(i) \to F(k)$ is unique such that 
${\cal G}_h  \equaldot Q \circ R$.
Moreover $f:F(i) \to F(j)$ is unique such that
${\cal G}_f  \equaldot R$, 
and $g:F(j') \to F(k)$ is unique such that
${\cal G}_g  \equaldot Q$. 
By (*) above we have
$$Q \circ R \equaldot  {\cal G}_g  \circ {\cal G}_f \equaldot {\cal G}_{g \circ F(p) \circ f}. $$
Hence $h = g \circ F(p) \circ f$ as required.

Suppose $N(i,i,I_{{\check F}(i)})=(i,i,f)$ where $f:F(i) \to F(i)$ is unique such that
$${\cal G}_f  \equaldot I_{{\check F}(i)}.$$
By (**) above
$${\cal G}_{{\rm id}_{F(i)}} \equaldot I_{{\check F}(i)}.$$
Hence $f= {\rm id}_{F(i)}$ as required.

 The functors $M$ and $N$ form an isomorphism of categories. This is clear for objects.
 Let $(i,j,R)$ be an arrow of ${\cal S}(I,F)$. Then
 $N(i,j,R)=(i,j,f)$ where $f:F(i) \to F(j)$ is unique such that ${\cal G}_f  \equaldot R.$ Now 
 $$M(N(i,j,R))=M(i,j,f)=(i,j, {\cal G}_f)=(i,j,R).$$
 Conversely 
 $$N(M(i,j,f))=N(i,j,G_f) = (i,j,f).$$
Thus we have established:

\begin{theorem} \label{mainiso}
${\cal S}(I,F) \cong {\cal C}(I,F)$ $\qed$
\end{theorem}

\section{Full images of categories in E-categories}
The construction of ${\cal C}(I,F)$ may actually be constructed as a full image of $F$ regarded as an
E-functor from $I$ (as discrete category) into the E-category of setoids. This follows from a general
full image construction (Theorem \ref{fullimage}). To prepare for a formal proof of this we need to present some more notions.

An equivalent formulation of category is the following (see Appendix for a proof of equivalence).
A {\em hom family presented category} $\bbC$ (or just HF-category) consists of a setoid ${\rm Ob}\;\bbC={\rm Ob}$ and a (proof-irrelevant) family ${\rm Hom}_{\bbC}$ of setoids indexed by the setoid ${\rm Ob} \times {\rm Ob}$. We often write, as is usual, $\bbC(a,b)$ for ${\rm Hom}_{\bbC}(a,b)$.
For each $a \in {\rm Ob}$, there is an element ${\rm id}_a \in {\rm Hom}(a,a)$.
Moreover for all $a,b,c \in {\rm Ob}$ there is an extensional function
$$\circ_{a,b,c} =\circ: {\rm Hom}(b,c) \times {\rm Hom}(a,b) \to {\rm Hom}(a,c).$$
These satisfies the usual equations of identity and associativity.
Moreover, for $p: a=_{\rm Ob}a'$,
\begin{equation} \label{coherenceid}
{\rm id}_{a'} = {\rm Hom}(p,p)({\rm id}_a)
\end{equation}
and for $p: a=_{\rm Ob}a'$, $q: b=_{\rm Ob}b'$ and $r: c=_{\rm Ob}c'$ this diagram commutes:
\begin{equation} \label{coherencecomp}
\bfig\square<1200,500>[{\rm Hom}(b,c) \times {\rm Hom}(a,b)`{\rm Hom}(a,c)`{\rm Hom}(b',c') \times 
{\rm Hom}(a',b')`{\rm Hom}(a',c');\circ_{a,b,c}`{\rm Hom}(q,r)\times {\rm Hom}(p,q)`{\rm Hom}(p,r)`\circ_{a',b',c'}]\efig
\end{equation}
The equations (\ref{coherenceid}) and (\ref{coherencecomp}) are {\em coherence conditions} for the transportation 
maps of the hom-family.

A weaker notion is that of a {\em E-category,} where we require in the above instead that ${\rm Ob}$ is type, and that ${\rm Hom}$ is family of setoids indexed by the type ${\rm Ob} \times {\rm Ob}$. Moreover
we drop equations (\ref{coherenceid}) and (\ref{coherencecomp}).
Any HF-category may be considered as an E-category by omitting the equality on objects.

A {\em functor $F$} from the HF-category $\bbC$ to the HF-category $\bbD$ consists of an extensional function $F_0: {\rm Ob}\; \bbC \to {\rm Ob}\; \bbD$ and for each
pair of objects $a,b \in {\rm Ob}\; \bbC$,  an extensional function $F_{a,b}: {\rm Hom}_{\bbC}(a,b) \to {\rm Hom}_{\bbD}(F_0(a),F_0(b))$
satisfying the usual functoriality equations.
Moreover it is required that for $p:a=_{{\rm Ob}\; \bbC} a'$, $q:b=_{{\rm Ob}\; \bbC}b'$, the diagram
\begin{equation} \label{coherencefunct}
\bfig\square<1200,500>[ {\rm Hom}(a,b)`{\rm Hom}(F_0(a),F_0(b))` {\rm Hom}(a',b')`{\rm Hom}(F_0(a'),F_0(b'));F_{a,b}` {\rm Hom}(p,q)`{\rm Hom}({\rm ext}(F_0,p),{\rm ext}(F_0,q))`F_{a',b'}]\efig
\end{equation}
commutes. Here ${\rm ext}(F_0,r)$ denotes the canonical proof that $F_0(c)=_{{\rm Ob}\; \bbD} F_0(c')$ for
$r: c =_{{\rm Ob}\; \bbC} c'$. (Because of the proof-irrelevance of ${\rm Hom}$, it does not matter what this proof object actually is the diagram above.)

For an {\em E-functor} between E-categories the condition that $F_0$ is extensional is omitted, and the 
coherence condition (\ref{coherencefunct}) is dropped.

\medskip
We may construct the full image of an E-functor as an HF-category if the source category is an HF-category.

\medskip
\begin{theorem} \label{fullimage}
Let $F:\bbC \to \bbD$ be an E-functor from an HF-category $\bbC$ to an E-category $\bbD$. 
Then for the HF-category $\bbS$ with objects ${\rm Ob}\; \bbS =_{\rm def} {\rm Ob}\; \bbC$ and 
$${\bbS}(a,b) =_{\rm def} {\bbD}(F(a),F(b)),$$
and where
$${\rm id}^{\bbS}_a =_{\rm def} {\rm id}^{\bbD}_{F(a)}, \qquad \circ_{a,b,c}^{\bbS}=_{\rm def} \circ_{F(a),F(b),F(c)}^{\bbD},$$
there is a functor $G: \bbC \to \bbS$ given by $G(a)=_{\rm def} a$ and $G_{a,b}(f)=_{\rm def} F_{a,b}(f)$
which is surjective on objects. The HF-category $\bbS$ is a full E-subcategory of $\bbD$.
\end{theorem}
\begin{proof}
It is clear that $\bbS$ is an E-category. We show it is an HF-category as well.
For $p:a=_{{\rm Ob}\; \bbS}a'$ and $q:b=_{{\rm Ob}\; \bbS} b'$, we need to define the transportation
map
$${\bbS}(p,q):{\bbS}(a,b) \to {\bbS}(a',b')$$
From the transportation maps of $\bbC$, we have ${\bbC}(p, r(a)): {\bbC}(a,a) \to {\bbC}(a',a)$
and 
 ${\bbC}(r(b),q): {\bbC}(b,b) \to {\bbC}(b,b')$
so ${\bbC}(p,r(a))({\rm id}_a) \in {\bbC}(a',a)$ and ${\bbC}(r(b),q)({\rm id}_b) \in {\bbC}(b,b').$
Thus let
 $$ {\bbS}(p,q)(f) = F({\bbC}(r(b),q)({\rm id}_b)) \circ f \circ F({\bbC}(p,r(a))({\rm id}_a)).$$
It is clear that ${\bbS}(p,q)$ is  extensional.  We also have ${\bbS}(p,q) =_{\rm ext} {\bbS}(p',q')$ for all $p,p':a=_{\bbS} a'$ and all $q,q': b=_{\bbS} b'$, since ${\bbC}$ is a proof-irrelevant family.
Moreover
$${\bbS}(r(a),r(b))(f) = {\rm id}_b \circ f \circ {\rm id}_a =f.$$
For $q: b=_{\bbS} b'$, $q': b'=_{\bbS} b''$,  $p: a=_{\bbS} a'$, $p': a'=_{\bbS} a''$, 
$$
 {\bbS}(p'\circ p,q'\circ q)(f) =  
F({\bbC}(r(b),q'\circ q)({\rm id}_b)) \circ f \circ F({\bbC}(p'\circ p,r(a))({\rm id}_a)).
$$

By using the coherence conditions for hom-setoids we obtain
\begin{eqnarray*} 
\lefteqn{{\bbC}(r(b'),q')({\rm id}_{b'}) \circ  
{\bbC}(r(b),q)({\rm id}_b)} \\
&\overset{(\ref{coherenceid})}{=}& 
{\bbC}(r(b'),q')({\bbC}(q,q)({\rm id}_b)) \circ  
{\bbC}(r(b),q)({\rm id}_b)   \\
&\overset{(F3)}{=}& 
{\bbC}(r(b')\circ q,q'\circ q)({\rm id}_b) \circ  
{\bbC}(r(b),q)({\rm id}_b)   \\
&=& 
{\bbC}(q,q'\circ q)({\rm id}_b) \circ  
{\bbC}(r(b),q)({\rm id}_b)   \\
&\overset{(\ref{coherencecomp})}{=}& 
{\bbC}(r(b),q'\circ q)({\rm id}_b \circ {\rm id}_b)  
= 
{\bbC}(r(b),q'\circ q)({\rm id}_b) 
\end{eqnarray*}
Similarly, 
\begin{eqnarray*}
\lefteqn{{\bbC}(p,r(a))({\rm id}_a) \circ  
{\bbC}(p',r(a'))({\rm id}_{a'})} \\
&\overset{(\ref{coherenceid})}{=}&{\bbC}(p,r(a))({\rm id}_a) \circ  
{\bbC}(p',r(a'))({\bbC}(p,p)({\rm id}_a)) \\
&\overset{(F3)}{=}&  {\bbC}(p,r(a))({\rm id}_a) \circ  {\bbC}(p'\circ p,r(a') \circ p)({\rm id}_a)  \\
&\overset{(\ref{coherencecomp})}{=}& {\bbC}(p'\circ p,r(a))({\rm id}_a) 
\end{eqnarray*}
Thus
\begin{eqnarray*}
\lefteqn{
 {\bbS}(p'\circ p,q'\circ q)(f)} \\
 &=&
 F{\bbC}(r(b'),q')({\rm id}_b')) \circ  
F({\bbC}(r(b),q)({\rm id}_b)) \circ \\
 && \qquad \qquad f \circ F({\bbC}(p,r(a))({\rm id}_a)) \circ  
F({\bbC}(p',r(a'))({\rm id}_{a'})) \\
&=& {\bbS}(p',q')({\bbS}(p,q)(f))
 \end{eqnarray*}
Hence ${\rm Hom}_{\bbS}$ is a proof-irrelevant family over ${\rm Ob}\; \bbS \times {\rm Ob}\; \bbS$.
The equations for identity and associativity are clearly fulfilled, since they are inherited from $\bbD$. The coherence conditions (\ref{coherenceid}) and (\ref{coherencecomp})
follows by functoriality of $F$: As for (\ref{coherenceid})
suppose $p: a=_{{\rm Ob}\; \bbS} a'$.
\begin{eqnarray*}
\bbS(p,p)({\rm id}^\bbS_a) &=&
F({\bbC}(r(a),p)({\rm id}_{a})) \circ {\rm id}^\bbD_{F(a)} \circ F({\bbC}(p,r(a))({\rm id}_a)) \\
&=&
F({\bbC}(r(a),p)({\rm id}_{a})) \circ F({\bbC}(p,r(a))({\rm id}_a)) \\
&=&
F({\bbC}(r(a),p)({\rm id}_{a}) \circ {\bbC}(p,r(a))({\rm id}_a))\\
&=&
F({\bbC}(p,p)({\rm id}_{a})) = F({\rm id}_{a'}) = {\rm id}^\bbD_{F(a')} = {\rm id}^\bbS_{a'}\
\end{eqnarray*}
Regarding the condition (\ref{coherencecomp}) suppose that
$p: a=_{{\rm Ob}\; \bbS}a'$, $q: b=_{{\rm Ob}\; \bbS}b'$ and $r: c=_{{\rm Ob}\; \bbS}c'$
and that $f \in \bbS(b,c)$ and $g \in \bbS(a,b)$,
\begin{eqnarray*}
\bbS(q,r)(f) \circ \bbS(p,q)(g) &=& F({\bbC}(r(c),r)({\rm id}_c)) \circ f \circ F({\bbC}(q,r(b))({\rm id}_b))
 \circ \\
 && \qquad \qquad F({\bbC}(r(b),q)({\rm id}_b)) \circ g \circ F({\bbC}(p,r(a))({\rm id}_a)) \\
 &=& F({\bbC}(r(c),r)({\rm id}_c)) \circ f \circ \\
 && \qquad \qquad F({\bbC}(q,r(b))({\rm id}_b) \circ {\bbC}(r(b),q)({\rm id}_b)) \circ g \circ F({\bbC}(p,r(a))({\rm id}_a)) \\
&=& F({\bbC}(r(c),r)({\rm id}_c)) \circ f \circ \\
 && \qquad \qquad F({\bbC}(q,q)({\rm id}_b)) \circ g \circ F({\bbC}(p,r(a))({\rm id}_a)) \\
&=& F({\bbC}(r(c),r)({\rm id}_c)) \circ f \circ g \circ F({\bbC}(p,r(a))({\rm id}_a)) \\
&=& \bbS(p,r)(f \circ g).
\end{eqnarray*}

 $G$ is evidently an E-functor surjective on objects. We check the coherence condition 
 (\ref{coherencefunct}): Suppose that $p:a=_{{\rm Ob}\; \bbC} a'$, $q:b=_{{\rm Ob}\; \bbC}b'$
 and that $f \in \bbC(a,b)$. Write $p'={\rm ext}(G,p)$ and $q'={\rm ext}(G,q)$.
 \begin{eqnarray*}
 \bbS(p',q')(G(f))&=& 
 F({\bbC}(r(b),q')({\rm id}_b)) \circ G_{a,b}(f) \circ F({\bbC}(p',r(a))({\rm id}_a))
 \\
 &=&  F({\bbC}(r(b),q')({\rm id}_b)) \circ F(f) \circ F({\bbC}(p',r(a))({\rm id}_a))  \\
 &=& F({\bbC}(r(b),q')({\rm id}_b) \circ f \circ {\bbC}(p',r(a))({\rm id}_a)) \\
 &=& F({\bbC}(r(b),q')({\rm id}_b) \circ \bbC(r(a),r(b))(f) \circ {\bbC}(p',r(a))({\rm id}_a)) \\
 &\overset{(\ref{coherencecomp})}{=}& F({\bbC}(r(a),q')({\rm id}_b \circ f) \circ {\bbC}(p',r(a))({\rm id}_a)) \\
  &\overset{(\ref{coherencecomp})}{=}& F({\bbC}(p',q')({\rm id}_b \circ f \circ {\rm id}_a)) \\
  &=& F({\bbC}(p',q')(f)) = G({\bbC}(p',q')(f))
 \end{eqnarray*}
 
\end{proof}

\begin{example} {\em Let $\bbC$ be discrete category arising from a setoid $A$ and let $\bbD={\bf Setoids}$ be the  E-category of setoids.  Suppose that $F$ is a proof-irrelevant family of setoids indexed by $A$. Then $F$ may be considered as an E-functor $\bbC \to {\bf Setoids}$, and the full image $\bbS$ is essentially ${\cal C}(A,F)$.
}
\end{example}

\section*{Appendix: Categories in type theory}

\subsection*{Essentially Algebraic Formulation}

Similarly to the standard set-theoretic definition, we define
in type theory a {\em category} $\bbC$ as a triple of setoids $\bbC_0$, $\bbC_1$, $\bbC_2$ consisting of  {\em objects,} {\em arrows} and {\em composable arrows,}
equipped with extensional functions ${\sf id}: \bbC_0 \to\, \bbC_1$, ${\sf dom}, {\sf cod}: \bbC_1 \to\, \bbC_0$ and
${\sf cmp}, {\sf fst}, {\sf snd}: \bbC_2 \to \bbC_1$ that satisfy the axioms

\begin{multicols}{2}
\begin{itemize}
\item[A1.] ${\sf dom}({\sf id} (x)) = x$,
\item[A2.] ${\sf cod}({\sf id} (x)) = x$,
\item[A3.] ${\sf dom}({\sf cmp}(u)) =
                    {\sf dom} ({\sf fst}(u))$, 
\item[A4.] ${\sf cod}({\sf cmp}(u)) =
                    {\sf cod} ({\sf snd}(u))$,
\end{itemize}
\end{multicols}

and

\begin{itemize}
\item[A5.] ${\sf fst}(u) = {\sf fst}(v),
                 {\sf snd}(u) ={\sf snd}(v) \implies u = v$,
\item[A6.] ${\sf dom}(f) = {\sf cod}(g) \implies
    \exists u \in \bbC_2 ({\sf snd}(u) = f  \land {\sf fst}(u) = g)$,
\item[A7.] ${\sf fst}(u) = {\sf id}(y) \implies
              {\sf cmp}(u) = {\sf snd}(u)$,
\item[A8.] ${\sf snd}(u) = {\sf id}(x) \implies
              {\sf cmp}(u) = {\sf fst}(u)$,
\item[A9.] ${\sf fst}(w) = {\sf fst}(v), {\sf snd}(v) = {\sf fst}(u), 
       {\sf snd}(u) = {\sf snd}(z),{\sf snd}(w) = {\sf cmp}(u), 
       {\sf cmp}(v) = {\sf fst}(z) \implies {\sf cmp}(w) = {\sf cmp}(z)$.
\end{itemize}

A {\em functor} $F: \bbB \to \bbC$ is a triple of extensional functions $F_k: \bbB_k \to \bbC_k$, $k=0,1,2$, such
that all operations of the categories are preserved, that is
\begin{multicols}{2}
\begin{itemize}
\item[] $F_1 \circ {\sf id} = {\sf id} \circ F_0$,
\item[] $F_0 \circ {\sf dom} = {\sf dom} \circ F_1$,
\item[] $F_0 \circ {\sf cod} = {\sf cod} \circ F_1$,
\item[] $F_1 \circ {\sf fst} = {\sf fst} \circ F_2$,
\item[] $F_1 \circ {\sf snd} = {\sf snd} \circ F_2$,
\item[] $F_1 \circ {\sf cmp} = {\sf cmp} \circ F_2$. 
\end{itemize}
\end{multicols}

\subsection*{Equivalence to the Hom Family Formulation}

Let $\bbC$ be a category formulated in the algebraic formulation. We define an HF-category ${\cal C}$. The objects of $\cal C$ is $\bbC_0$. For $a, b \in {\cal C}$ define the setoid
$${\rm Hom}_{\cal C}(a,b) = (\Sigma f \in \bbC_1.{\sf dom}(f)=_{\bbC_0}a \land {\sf cod}(f) =_{\bbC_0} b, \sim)$$
where $(f,r) \sim (f',r')$ if and only if $f=_{\bbC_1} f'$.  For $p:a =_{\bbC_0} a'$ and $q:b =_{\bbC_0} b'$,
define an extensional function
$${\rm Hom}(p,q): {\rm Hom}(a,b) \to {\rm Hom}(a',b')$$
by letting
$${\rm Hom}(p,q)(f,r) =(f,r'),$$
where $r'$ is some proof of ${\rm dom}(f)=a' \land {\rm cod}(f) = b'$ obtained from $r$, $p$ and $g$. As the second component $r'$ is irrelevant, ${\rm Hom}$ is a proof-irrelevant family of setoids on $\bbC_0 \times \bbC_0$.

For $a \in \bbC_0$, let ${\rm id}_{a} =({\sf id}(a),r)$ where $r$ is some proof that ${\sf dom}({\sf id}(a))=_{\bbC_0}a \land {\sf cod}({\sf id}(a)) =_{\bbC_0} a$. This uses (A1) and (A2). For $p:a=_{\bbC_0} a'$
$${\rm Hom}(p,p)({\rm id}_a) \sim {\rm Hom}(p,p)({\sf id}(a),r) \sim  ({\sf id}(a),r'') \sim ({\sf id}(a'),r') \sim {\rm id}_{a'}  $$
as required by (\ref{coherenceid}).
Define composition 
$$\circ:{\rm Hom}(b,c) \times {\rm Hom}(a,b) \to {\rm Hom}(a,c)$$ 
as follows. For $(g,r)\in {\rm Hom}(b,c)$ and $(f,s)\in {\rm Hom}(a,b)$, we have ${\sf cod}(f) = {\sf dom}(g) =b$. By (A5) and (A6) there
is a unique $u \in \bbC_2$ such that ${\sf snd}(u)=g$ and ${\sf fat}(u)=f$. Let $h= {\sf cmp}(u)$. By (A3) and (A4) follows then ${\sf dom}(h)= a$ and ${\sf cod}(h) = c$. Hence $(h,r') \in {\rm Hom}(a,c)$ for some $r'$ (which is irrelevant).
Let thus 
$$(g,r) \circ_{a,b,c} (f,s) =_{\rm def} (h,r').$$
Since the second components are irrelevant (H1) -- (H3) below follows easily from (A7) -- (A9).
 \begin{itemize}
\item[(H1)] ${\rm id}_b \circ f =_{{\rm Hom}(a,b)} f$ for $f \in {\rm Hom}(a,b)$,
\item[(H2)] $f  \circ {\rm id}_a=_{{\rm Hom}(a,b)} f$ for $f \in {\rm Hom}(a,b)$,
\item[(H3)] $f \circ (g \circ h) =_{{\rm Hom}(a,d)} (f \circ g) \circ h$ for $h \in {\rm Hom}(a,b)$, $g \in {\rm Hom}(b,c)$ 
and $f \in {\rm Hom}(a,b)$.
\end{itemize}
The irrelevance property of the second component also entails (\ref{coherencecomp}).

\medskip
Conversely suppose that ${\cal C}$ is an HF-category. Define an essentially algebraic category $\bbC$, by letting  $\bbC_0 = {\rm Ob}\;{\cal C}$. Then define $\bbC_1$ to be the setoid consisting of triples
$$(a,b,f)$$
where $f \in {\rm Hom}_{\cal C}(a,b)$. Define a relation $\sim$ by letting
$$(a,b,f) \sim (a',b',f') \Longleftrightarrow_{\rm def} (\exists p:a=_{\bbC_0} a')(\exists q:b =_{\bbC_0} b') {\rm Hom}_{\cal C}(p,q)(f) = f'.$$
This is an equivalence relation since ${\rm Hom}_{\cal C}$ is a proof-irrelevant family. 
Define ${\sf dom}(a,b,f) = a$ and ${\sf cod}(a,b,f) = b$.

The setoid  $\bbC_2$ of composable maps consists of triples 
$$({\bf f}, {\bf g},p)$$
where ${\bf f} \in \bbC_1$, ${\bf g} \in \bbC_1$ and $p:{\sf cod}({\bf f})=_{\bbC_0} {\sf dom}({\bf g})$.
Define 
$$({\bf f}, {\bf g},p)=_{\bbC_2} ({\bf f'}, {\bf g'},p') \Longleftrightarrow_{\rm def} {\bf f}=_{\bbC_1} {\bf f'} \land 
{\bf g}=_{\bbC_1}  {\bf g'}.$$
Define ${\sf cmp}: \bbC_2 \to \bbC_1$ as follows
$${\sf cmp}((a,b,f),(c,d,g),p) = (a,d, g \circ F(p) \circ f).$$
Here $F(p) = {\rm Hom}(r(a),p)({\rm id}_a)$, where $r(a): a=_{\bbC_0}a$.
The conditions (A1) -- (A9) may be verified straightforwardly using identities such as
$${\rm Hom}(p,q)({\rm id}_a) \circ g = {\rm Hom}(q^{-1}\circ p,{\rm r}(b))(g),$$
$$f \circ {\rm Hom}(p,q)({\rm id}_a) = {\rm Hom}(p\circ q^{-1},{\rm r}(b))(f).$$
and the fact that ${\rm Hom}$ is a proof irrelevant family.


\begin{thebibliography}{99}
\bibitem{BCP} Gilles Barthe, Venanzio Capretta and Olivier Pons. Setoids in type theory. {\em Journal of Functional Programming} 13(2003), 261--293.
\bibitem{P12} Erik Palmgren. Proof-relevance of families of setoids and identity in type theory. {\em Archive for Mathematical Logic} 51(2012), 35--47.
\bibitem{PW} Erik Palmgren and Olov Wilander. {\em Constructing categories and setoids of setoids in type theory.} Preprint March 2013.
\bibitem{W} Olov Wilander. Constructing a small category of setoids. {\em Mathematical Structures in Computer Science} 22(2012), pp.\ 103 -- 121.
\end{thebibliography}
\end{document}